\documentclass[12pt]{article}
\usepackage[utf8]{inputenc}
\usepackage{graphicx}

\usepackage{mathtools, bm}
\usepackage{amssymb, bm}
\usepackage{amsthm}
\usepackage{amsfonts}
\usepackage{stmaryrd}
\usepackage{hyperref, enumerate}
\usepackage{amsmath}
\usepackage[T1]{fontenc}
\usepackage{setspace}
\usepackage{dsfont}
\usepackage{array, color}
\usepackage{fancybox}

\newcommand{\Z}{\mathbb{Z}}

\newcommand{\Q}{\mathbb{Q}}

\newtheorem{lemme}{Lemma}[section]
\newtheorem{coro}{Corollary}[section]

\newtheorem{rmq}{Remark}[section]

\newtheorem{thm}{Theorem}[section]

\begin{document}

\title{Explicit lower bounds for the height in Galois extensions of number fields \footnote{This work has been partially supported by the Institut Fourier, Université Grenoble1, UMR 5582 duCNRS, 100 rue des mathématiques, BP 74, 38402 St Martin d’hères.}} 
\author{Jonathan Jenvrin}
\date{}

\maketitle

\begin{abstract}
    Amoroso and Masser proved that for every real $\epsilon > 0$, there exists a constant $c(\epsilon)>0$, with the property that, for every algebraic number $\alpha$ such that $\Q(\alpha)/\Q$ is a Galois extension, the height of $\alpha$ is either 0 or at least $c(\epsilon)  [\Q(\alpha):\Q]^{-\epsilon}$. In the present article, we establish an explicit version of the aforementioned theorem.
\end{abstract}

\section{Introduction}
In this article, we let $\overline{\Q}$ be a fixed algebraic closure of $\Q$.
For an algebraic number $\alpha$ of degree $d$ over $\Q$, we denote by $h(\alpha)$ its absolute logarithmic Weil height defined as
$$h(\alpha) = \frac{1}{d} \left( \log \left(|a| \prod_{i=1}^{d} \max\left(1, |\alpha_{i}|\right) \right) \right)$$
where $a$ is the leading coefficient of the minimal polynomial of $\alpha$ over $\mathbb{Z}$, and $\alpha_{1}, \dots, \alpha_{d}$ are the conjugates of $\alpha$ over $\Q$. Notice in particular that $h(\alpha) \ge 0$ for all $\alpha \in \overline{\Q}.$

While, by Kronecker's theorem (see for instance \cite[Theorem 1.5.9]{BombieriGubler}), it is well-known that $h(\alpha) = 0$ if and only if $\alpha$ is either $0$ or a root of unity, in \cite{Lehmer} Lehmer raised the question of whether there exists a constant $c > 0$ such that
$$ h(\alpha) \ge \frac{c}{d}$$
whenever $h(\alpha)$ is not zero.
The existence of such a constant is nowadays known as Lehmer's conjecture and has been proved for various classes of algebraic numbers, but is still open in general. For instance, the conjecture is obviously true for $\alpha$ not a unit with $c=\log(2)$. While for $\alpha$ non-reciprocal (which is always the case for $d$ odd) its validity was proved in \cite{SmythCJ} with $c=3h(\theta)=\log(\theta)$, where $\theta$ is the real root $>1$ of $X^{3}-X-1$. The most notable progress toward Lehmer's conjecture is Dobrowolski's result \cite[Theorem 1]{[12]}, later made explicit by Voutier in \cite[Theorem on p. 83]{voutier}, proving that if $h(\alpha) \neq 0$, then
\[ h(\alpha) \ge \frac{1}{4d} \left( \frac{\log \log d}{\log d} \right)^{3}. \]
This implies that for any $\epsilon > 0$, there is an explicit constant $\tilde{c}(\epsilon) > 0$ such that either $h(\alpha) = 0$ or $h(\alpha) \ge \tilde{c}(\epsilon) d^{-1-\epsilon}$.

While Dobrowolski's result remains the only unconditional one on this problem, it is possible to prove that specific classes of algebraic numbers satisfy even stronger variants of Lehmer's conjecture, such as the Bogomolov property introduced by Bombieri and Zannier in \cite{bombieriZannier2001note}. A set of algebraic numbers \( S \) has the \emph{Bogomolov property (B)} if there exists a constant \( c=c(S)>0 \) such that for every \( \alpha \in S \), either \( h(\alpha)=0 \) or \( h(\alpha) \geq c \).

A set of algebraic numbers that has garnered attention in recent years is that of all \( \alpha \in \overline{\Q} \) such that \( \Q(\alpha)/\Q \) is Galois.

Amoroso and David proved that Lehmer's conjecture holds for elements of this set (see \cite[Corollary 1.7]{[7]}). Later, Amoroso and Masser \cite[Theorem 3.3]{[1]} showed an even stronger result: for any \( \epsilon > 0 \), there exists a positive effective constant \( c(\epsilon) \) such that, for every \( \alpha \in \overline{\Q} \) of degree \( d \) over \( \Q \) and not a root of unity, such that \( \Q(\alpha)/\Q \) is Galois, one has
\begin{align*}
  \label{thm-main}
  \tag{$\star$}  
  h(\alpha) \ge c(\epsilon) d^{-\epsilon}.
\end{align*}

This strong result raises the question of whether the set of algebraic numbers that generate a Galois extension over $\Q$ satisfies Property (B). A natural way to tackle this question is to fix the Galois group \( G \) of \( \mathbb{Q}(\alpha)/\mathbb{Q} \). The answer is positive when \( G \) is abelian \cite{[13]}, dihedral \cite[Corollary 1.3]{Amo_Zannier-DihedralCase}, has an exponent bounded by an absolute constant \cite[Corollary 1.7]{Amoroso_David_Zannier-OnFieldWithPropertyB} or has odd order (since the field of totally real numbers satisfies Property (B) by \cite[Corollary 1]{[10]}). An even stronger result has been proven for some classes of generators of Galois extensions of $\Q$ of group $\mathfrak{S}_{n}$ in \cite[Theorem 1.2 and Theorem 1.3]{amoroso2018mahler}, or \( \mathfrak{A}_n \) in \cite[Theorem 1.4 and Theorem 1.5]{jonathanAlternatinggroup}; in these cases the height of such generators goes to infinity with $n$. All these results give evidences for a positive answer to the aforementioned question, which is still open in general.

The goal of our article is to give an explicit version of (\ref{thm-main}). Our main result is the following:

\begin{thm}\label{main thm}
    Suppose $\alpha \in \overline{\mathbb{Q}}^{*}$ is of degree $d$ over $\Q$. If $\alpha$ is not a root of unity and $\Q(\alpha)/\Q$ is Galois, then \begin{equation*}
     h(\alpha) \ge 10^{-8}\exp{ \left(-\frac{49}{2} \log(3d)^{3/4}\log(\log(3d))   \right)}.
\end{equation*}

\end{thm}

As an easy corollary, we obtain an explicit version of Amoroso and Masser's result \cite[Theorem 3.3]{[1]}.

\begin{coro} \label{explicitation de c(epsilon)}
    For every $\epsilon>0$ and for every $\alpha$ of degree $d$, not a root of unity, such that $\Q(\alpha)/\Q$ is Galois, one has $$h(\alpha) \ge c(\epsilon) d^{-\epsilon}$$
    where $$c(\epsilon)=10^{-8}\left( \frac{1}{3} \right)^{\epsilon} \exp{\left( -181 \left( \frac{724}{5\epsilon}\right)^{4}- \left( \frac{724}{5\epsilon}\right)^{5} \right)}.$$
    
\end{coro}

Theorem \ref{main thm} and Corollary \ref{explicitation de c(epsilon)} are proved in Section 3. Our proof strategy relies on that of Amoroso and Masser's result in \cite[Theorem 3.3]{[1]}. 

In particular, we divide our proof in two cases, according to the relative magnitude of the multiplicative rank $\rho(\alpha)$ of the subgroup of $\overline{\mathbb{Q}}^{\times}$ generated by the conjugates of $\alpha$, and the quantity $\log (3 \operatorname{deg}(\alpha))^{1 / 4}$, where $d(\alpha)=[\Q(\alpha):\Q]$. When $\rho(\alpha)>\log (3 \operatorname{deg}(\alpha))^{1 / 4}$, we conclude using a result of Amoroso and Viada \cite{[3]}, which we recall in Theorem \ref{thm in rank r}, bounding from below the products of heights of multiplicatively independent algebraic numbers. On the other hand, when $\rho(\alpha) \leq \log (3 \operatorname{deg}(\alpha))^{1 / 4}$ we conclude by applying a result of Amoroso and Delsinne \cite{[2]}, providing an explicit relative version of Dobrowolski's lower bound, where the degree $\operatorname{deg}(\alpha)=[\mathbb{Q}(\alpha): \mathbb{Q}]$ is replaced by a relative degree $[L(\alpha): L]$, where $L / \mathbb{Q}$ is a finite abelian extension. This is in contrast with Amoroso and Masser's proof of (\ref{thm-main}), which uses an older higher-dimensional generalization of Dobrowolski's lower bound proven by Amoroso and David \cite{[7]} when $\rho(\alpha)>\log (3 d(\alpha))^{1 / 4}$, and the relative Dobrowolski lower bound proven by Amoroso and Zannier when $\rho(\alpha) \leq \log (3 d(\alpha))^{1 / 4}$.

\section{Auxiliary Results} \label{auxiliary results}

We state  two important results that will be the key ingredients in our proof. The first was proved by Amoroso and Viada in \cite[Corollary 1.6]{[3]}, where, more generally, they give an explicit version of a generalized
Dobrowolski result on Lehmer's problem.

\begin{thm}[{\cite[Corollary 1.6]{[3]}}] \label{thm in rank r}
    Let $\alpha_{1}, \dots, \alpha_{n}$ be multiplicatively independent algebraic numbers in a number field K. Then $$h(\alpha_{1}) \dots h(\alpha_{n}) \ge [K:\Q]^{-1}(1050n^{5}\log(3[K:\Q]))^{-n^{2}(n+1)^{2}}.$$
\end{thm}

The second result was proved by Amoroso and Delsinne in \cite[Théorème 1.3]{[2]}, and provides a relative version of Dobrowolski's lower bound, when the base field considered is abelian. We recall a special case of this theorem, which is enough for our purposes.

\begin{thm}[{\cite[Théorème 1.3]{[2]}}] \label{thm in rank 1}
        Let $\alpha \in \overline{\Q}^{*}$ be not a root of unity and let $L$ be a number field. Then if  $L/\Q$ is a finite abelian extension and $D=[L(\alpha):L]$, we have \[h(\alpha) \ge D^{-1}\frac{\log\log(5D)^{3}}{\log(2D)^{4}}.\]
\end{thm}

\begin{proof}
    We apply \cite[Théorème 1.3]{[2]}, with $\mathbb{L}=L$ and $\mathbb{K}=\Q$. Notice that the lower bound in \cite{[2]} depends on a certain quantity $\left( g([\mathbb{K}:\Q])\cdot \Delta_{\mathbb{K}} \right)^{-c}$ which is $1$ for $\mathbb{K}=\Q$. This gives the desired result.

\end{proof}

The following lemma is a relative version of \cite[Lemma 2.2]{[1]}.

 \begin{lemme}  \label{3^rho^2}
Let $F/ K$ be a finite Galois extension and $\alpha \in F\setminus \{0\}$. Let $\alpha_{1},\ldots, \alpha_{d}$ be the conjugates of $\alpha$ over $K$, $i.e.$ the orbit of $\alpha$ under the action of $\mathrm{Gal}(F/K)$. Moreover let $\rho$ be the rank of the multiplicative group generated by $\alpha_{1},\ldots,\alpha_{d}$, and suppose that $\rho \ge 1$. Then there exists a subfield $L \subset F$ which is Galois over $K$ of degree $[L: K] \le n(\rho) $, such that $F$ contains a primitive $k$-th root of unity $\zeta_{k}$ and $\alpha^{k} \in L$, where $k$ is the order of the group of roots of unity in $F$.
We can take $$n(\rho)=\rho ! 2^{\rho} \mbox{ for } \rho=1,3,5 \mbox{ and } \rho >10. $$
Otherwise we have $$
\begin{array}{l|lllllll}
n(\rho) & 2 & 4 & 6 & 7 & 8 & 9 & 10 \\
\hline \rho & 12 & 1152 & 103680 & 2903040 & 696729600 & 1393459200 & 8360755200
\end{array}
$$
\end{lemme}
 
\begin{proof}
Define $\beta_{i}=\alpha_{i}^{k}$ for $1\le i \le d$ and $L=K(\beta_{1},\dots,\beta_{d})$. We have $L \subset F$ because $F/K$ is Galois, and we easily check that $L/K$ is Galois. The $\Z-$module $$M=\{\beta_{1}^{a_{1}}\dots \beta_{d}^{a_{d}} \; | \; a_{1},\dots,a_{d} \in \Z \}$$ is torsion free by the choice of $k$ and so, by the classification of finitely generated abelian groups (\cite[Theorem 8.1 and Theorem 8.2]{Lang}), is free, of rank $\rho$. This shows that the action of $\mathrm{Gal}(L/K)$ over $M$ defines an injective representation $\mathrm{Gal}(L/K) \rightarrow \mathrm{GL}_{\rho}(\Z)$. We obtain that $\mathrm{Gal}(L/ K)$ identifies to a finite subgroup of $\mathrm{GL}_{\rho}(\Z)$. To conclude, we use a theorem stated by Feit in 1996, and proved by Rémond in \cite[Théorème 7.1]{GaelRemond}, which computes the largest cardinality of a finite subgroup of $\mathrm{GL}_{\rho}(\Z)$.
\end{proof}

\begin{rmq}
    The constant $n(\rho)$ in Lemma \ref{3^rho^2} is somehow optimal, since it is equal to the largest cardinality of a finite subgroup of $\mathrm{GL}_{\rho}(\Z)$.
\end{rmq}

We continue this section with the following lemma, which gives an explicit upper bound for the quotient of an integer over his Euler's totient.

\begin{lemme} \label{minoration explicite de l'indicatrice d'euler}
Let $\phi$ be the Euler's totient function. For every positive integer $n \ge 1$, we have  $$\frac{n}{\phi(n)} \le \frac{\log(\log(3n))}{\log(\log(3))}$$
and $$\phi(n) \ge \sqrt{\frac{n}{2}}.$$
    
\end{lemme}

\begin{proof}
The first inequality can be easily deduced from \cite[Theorem 15]{Phi(n)<sqrt(n/2)} for \( n \ge 100 \) and verified for smaller values of \( n \). Indeed, if we denote by \( \gamma \) the Euler's constant, then for \( n \ge 100 \), we have
\[ 
e^{\gamma} \log(\log(n)) + \frac{2.50637}{\log(\log(n))} \le 3 \log(\log(3n)) \le \frac{\log(\log(3n))}{\log(\log(3))}. 
\]
The second inequality is well-known, and can be deduced easily from the previous one.

\end{proof}

We conclude this section by recalling Stirling's upper bound (\cite[(1) and (2) on p. 26]{StirlingBounds}), which will be use in the next section.

\begin{lemme} \label{Stirling upper bound}
    For every $n\ge1$, we have $$\frac{n!}{n^{n}} \le \sqrt{2\pi n} e^{\frac{1}{12n}}e^{-n}.$$
\end{lemme}

\section{Proof of Theorem \ref{main thm}} \label{proof of main thm}
We can now establish the proof of Theorem \ref{main thm}. To this end, we fix for the rest of this section an algebraic number $\alpha \in \overline{\Q}$ such that $\Q(\alpha)/\Q$ is Galois.
    Firstly, we will prove a lower bound for $h(\alpha)$ depending on the multiplicative rank $\rho$ of the conjugates of $\alpha$, and on the degree $d$ of $\Q(\alpha)$ over $\Q$.

    \begin{lemme} \label{minoration avec dépendance en rho}
        Let 
        $$
        g_{1}(\rho,d) = \min_{1\le r \le \rho} \left( d^{1/r} \left(1050 r^{5}\log(3d) \right)^{r(r+1)^{2}} \right)
        $$
        and
        $$
        g_{2}(\rho,d) = 6.5 \cdot 10^{7} \rho^{\rho+5} \log \left(\log \left(6d^{2} \right) \right)^{5}.
        $$
        Then:
        $$
        h(\alpha)^{-1} \le \min \left( g_{1}(\rho,d), g_{2}(\rho,d) \right).
        $$
    \end{lemme}

   \begin{proof}
    Let $\alpha_{1},\ldots,\alpha_{\rho}$ be multiplicatively independent conjugates of $\alpha$, and fix $r \in \{1,\ldots,\rho \}$. Since $\Q(\alpha)/\Q$ is Galois by assumption, we have that $\alpha_{1},\ldots,\alpha_{r} \in \Q(\alpha)$. Therefore, applying Theorem \ref{thm in rank r} to $\alpha_{1},\ldots,\alpha_{r}$ we see that 
   
    \begin{equation*} 
        h(\alpha)=(h(\alpha_{1})\cdots h(\alpha_{r}))^{1/r} \ge d^{-1/r} \left(1050 r^{5}\log(3d) \right)^{-r(r+1)^{2}},
    \end{equation*}
    which implies that $h(\alpha)\ge g_{1}(\rho,d)^{-1}$.

    Let $k$ be the number of roots of unity in $\Q(\alpha)$. Then, by Lemma \ref{3^rho^2} there exists a subfield $L \subset \Q(\alpha)$ whose degree can be bounded as follows: $$[L:\Q] \le 135 \rho! 2^{\rho-1},$$ such that the extension $L/\Q$ is Galois and $\alpha^{k} \in L$. Set $M=\Q(\zeta_{k})$. We have $M(\alpha)=\Q(\alpha)=L(\alpha)$. Notice that
    \begin{align*}
        [M(\alpha):M] &= [L(\alpha):L] \frac{[L:\Q]}{[M:\Q]}.
    \end{align*}

    Since $\alpha$ is a root of $X^{k}-\alpha^{k} \in L[X]$, we have that $[L(\alpha):L] \le k$. We also notice that $\phi(k) =[M:\Q]$. Therefore, we obtain the following chain of inequalities:
    $$[M(\alpha):M] \le k\frac{[L:\Q]}{[M:\Q]} \le \frac{k}{\phi(k)}[L: \Q].$$

    By Lemma \ref{minoration explicite de l'indicatrice d'euler}, we have that
    $$\frac{k}{\phi(k)} \le \frac{\log(\log(3k))}{\log(\log(3))},$$
    so we obtain
    \begin{equation*}
        [M(\alpha):M] \le 135\frac{\log(\log(3k))}{\log(\log(3))} \rho! 2^{\rho-1}.
    \end{equation*}

    Since $k$ is the number of roots of unity contained in $L$, we have that $\phi(k) \le d$.
By the last statement of Lemma \ref{minoration explicite de l'indicatrice d'euler}, we have that
    $\sqrt{k/2} \le \phi(k)$, hence
    $k \le 2d^{2}$ and
    \begin{align*}
        [M(\alpha):M] &\le 135\frac{\log \left(\log \left(6d^{2} \right) \right)}{\log(\log(3))}  \rho! 2^{\rho-1}
        \le 718 \rho ! 2^{\rho}\log \left(\log \left(6d^{2} \right) \right).
    \end{align*}

    By Lemma \ref{Stirling upper bound}, applied to $\rho$, we have that
    \begin{align*}
[M(\alpha):M] &\le 718\rho^{\rho} \sqrt{2 \pi \rho} e^{\frac{1}{12 \rho}}e^{-\rho} 2^{\rho} \log \left(\log \left(6d^{2} \right) \right) \\
& \le 1440 \rho^{\rho} \log \left(\log \left(6d^{2} \right) \right)
\end{align*}

    By Theorem \ref{thm in rank 1}, applied to $\Q(\zeta_{k})/ \Q$, we have
    \begin{equation*}
        h(\alpha) \ge [M(\alpha):M]^{-1}\frac{\log \left(\log \left(5[M(\alpha):M] \right) \right)^{3}}{\log(2[M(\alpha):M])^{4}}.
    \end{equation*}

    Since
    $$x \mapsto \frac{1}{x} \frac{\log(\log(5x))^{3}}{\log(2x)^{4}}$$
    is a decreasing map on $[1,+\infty[$, we deduce, by setting
    $$X(\rho,d)=1440 \rho^{\rho} \log \left(\log \left(6d^{2} \right) \right) $$
    that
    \begin{align*}
        h(\alpha)^{-1} &\le X(\rho,d)\log(2X(\rho,d))^{4} \\
 &= 1440 \rho^{\rho} \log \left(\log \left(6d^{2} \right) \right) \log \left( 4884 \rho^{\rho} \log \left(\log \left(6d^{2} \right) \right) \right)^{4} \\
 &\le  12448 \rho^{\rho+4} \log \left(\log \left(6d^{2} \right) \right)^{5} \log \left( 4884^{1/\rho} \rho \right)^{4} \\
 & \le 6.5 \cdot 10^{7} \rho^{\rho+5} \log \left(\log \left(6d^{2} \right) \right)^{5}
\end{align*}
    This gives the second desired lower bound for $h(\alpha)$, and therefore proves the lemma.
\end{proof}

\begin{proof}[Proof of Theorem \ref{main thm}] Now, we are ready to prove Theorem \ref{main thm}. Suppose, first, that $\rho \leq \log(3d)^{1/4}$. By Lemma \ref{minoration avec dépendance en rho}, we get
\begin{align*} 
    h(\alpha)^{-1} &\le g_{2}(\rho,d)\le 6.5 \cdot 10^{7} \left(\log(3d)^{1/4} \right)^{\log(3d)^{1/4}+5} \log \left(\log \left(6d^{2} \right) \right)^{5} \\
    &\le 6.5 \cdot 10^{7}\exp \left( \left(\log(3d)^{1/4}+5 \right) \log \left(\log(3d)^{1/4} \right)+5\log \left(\log \left(\log \left(6d^{2} \right) \right) \right)  \right).\\
    &\le 6.5 \cdot 10^{7} \exp \left(\frac{3}{2}\log(3d)^{1/4} \log(\log(3d))+2 \log(3d)^{1/4} \log(\log(3d))  \right).
\end{align*}
So, we have  \begin{equation} \label{premier cas epsilon<1}
    h(\alpha)^{-1} \le 6.5 \cdot 10^{7}\exp \left( \frac{7}{2}\log(3d)^{1/4} \log(\log(3d))  \right).
\end{equation}

 Now, we suppose that $\rho \ge \log(3d)^{1/4}$. We let $r= \left\lceil \log(3d)^{1/4} \right\rceil$. In particular $r \leq \rho$. Then, thanks again to Lemma \ref{minoration avec dépendance en rho}, we obtain
\begin{align*}
    h(\alpha)^{-1} &\le g_{1}(\rho,d)= d^{1/r} \left(1050 r^{5}\log(3d) \right)^{r(r+1)^{2}} \\
    &\le d^{1/\log(3d)^{1/4}} \left(1050 \left(\log(3d)^{1/4}+1 \right)^{5}\log(3d) \right)^{\left(\log(3d)^{1/4}+1 \right)\left(\log(3d)^{1/4}+2 \right)^{2}}\\
    &\le d^{1/\log(3d)^{1/4}} \left(31693\log(3d)^{9/4} \right)^{6\log(3d)^{3/4}} \\
    &\le  \exp{ \left( \log(3d)^{3/4}+6\log(3d)^{3/4}\log \left(31693\log(3d)^{9/4} \right) \right)}.
\end{align*} 

So, we have  
\begin{equation} \label{2eme cas epsilon<1}
    h(\alpha)^{-1} \le  31693 \exp{ \left(\log(3d)^{3/4}+\frac{27}{2}\log(3d)^{3/4}\log(\log(3d)) \right)}.
\end{equation}

 From \eqref{premier cas epsilon<1} and \eqref{2eme cas epsilon<1}, we obtain
\begin{equation*}
     h(\alpha)^{-1} \le 6.5 \cdot 10^{7}\exp{ \left(\frac{49}{2} \log(3d)^{3/4}\log(\log(3d))   \right)}.
\end{equation*}
\end{proof}

\begin{proof}[Proof of Corollary \ref{explicitation de c(epsilon)}.] 
We fix $\epsilon > 0$. Notice that the function $$f(x)=\frac{\log(\log(3x))}{\log(3x)^{1/20}}$$ has a maximum at $$x=\frac{e^{e^{20}}}{3}.$$ Hence, we have that 
$$
\log(\log(3d)) \le \frac{20}{e} \log(3d)^{1/20},
$$
which implies 
\begin{equation} \label{equation1}
    \frac{49}{2} \log(3d)^{3/4}\log(\log(3d)) - \epsilon \log(d) \le 181 \log(3d)^{4/5} - \epsilon \log(d).
\end{equation}

Also, the function $f_{\epsilon}(x) = 181 \log(3x)^{4/5} - \epsilon \log(x)$ has its maximum at 
$$
x = \frac{1}{3}\exp{\left( \frac{724}{5\epsilon} \right)^{5}}.
$$
Therefore, we have 
\begin{equation} \label{equation2}
181 \log(3d)^{4/5} - \epsilon \log(d) \le 181 \left( \frac{724}{5\epsilon}\right)^{4} - \epsilon \log(1/3) - \left( \frac{724}{5\epsilon}\right)^{5}.
\end{equation}
Hence, by \eqref{equation1} and \eqref{equation2}, we have that 
$$
-\frac{49}{2} \log(3d)^{3/4}\log(\log(3d)) \ge -181 \left( \frac{724}{5\epsilon}\right)^{4} +\epsilon \log(1/3) + \left( \frac{724}{5\epsilon}\right)^{5}-\epsilon \log(d).
$$
and we conclude by Theorem \ref{main thm}.
  
\end{proof}

\section*{Acknowledgments} 
I am grateful to Gaël Rémond for comments and suggestions on an earlier draft of this article, and for pointing out the reference \cite{GaelRemond}.

\noindent I am also very grateful to the anonymous referee for the numerous insightful remarks and suggestions, which improved the quality of this paper. In particular I thank the referee for pointing out reference \cite{Phi(n)<sqrt(n/2)} which improved Lemma \ref{minoration explicite de l'indicatrice d'euler}.

\bibliographystyle{alpha}

JONATHAN JENVRIN: Univ. Grenoble Alpes, CNRS, IF, 38000 Grenoble, France

\textit{E-mail adress :} \href{mailto:jonathan.jenvrin@univ-grenoble-alpes.fr}{\texttt{jonathan.jenvrin@univ-grenoble-alpes.fr}}

\end{document}